\documentclass{amsart}
\usepackage{graphicx}
\usepackage{ulem}
\usepackage{bm}
\usepackage{color}
\usepackage{amssymb, amsmath, amsthm}
\usepackage{bm}
\theoremstyle{plain}
\usepackage{graphicx}
\usepackage[aboveskip=-5pt,position=bottom]{caption}
\usepackage{natbib}
\usepackage{url}
\usepackage{enumerate}
\usepackage{rotating}

\usepackage{accents}
\newcommand\tbar[1]{\accentset{\rule{.4em}{.8pt}}{#1}}
\newcommand\ubar[1]{\underaccent{\bar}{#1}}

\newtheorem{lemma}{Lemma}[section]

\newtheorem{theorem}[lemma]{Theorem}
\newtheorem{cor}[lemma]{Corollary}
\newtheorem{prop}[lemma]{Proposition}
\newtheorem{exam}[lemma]{\normalfont \scshape
 Example}
\newtheorem{rem}[lemma]{\normalfont \scshape Remark}

\newcommand{\R}{\mathbb{R}}
\newcommand{\N}{\mathbb{N}}

\newcommand{\Z}{\mathbb{Z}}
\newcommand{\norm}[1]{\left\Vert#1\right\Vert}
\newcommand{\dnormf}[1]{\wr\!\!\wr#1\wr\!\!\wr}
\newcommand{\abs}[1]{\left\vert#1\right\vert}
\newcommand{\set}[1]{\left\{#1\right\}}

\newcommand{\eps}{\varepsilon}

\newcommand{\bfx}{\bm{x}}
\newcommand{\bfzero}{\bm{0}}
\newcommand{\bfinfty}{\bm{\infty}}

\newcommand{\bfone}{\bm{1}}
\newcommand{\bfa}{\bm{a}}
\newcommand{\bfb}{\bm{b}}

\newcommand{\bfU}{\bm{U}}
\newcommand{\bfu}{\bm{u}}

\newcommand{\bfX}{\bm{X}}

\newcommand{\bfy}{\bm{y}}

\newcommand{\bfeta}{\bm{\eta}}

\newcommand{\bfm}{\bm{m}}

\newcommand{\bE}{\bar E^-(S)}
\newcommand{\bC}{\bar C^-(S)}

\renewcommand{\d}{\mathrm{d}}

\DeclareMathOperator{\Geom}{Geom}

\usepackage[doublespacing]{setspace}

\allowdisplaybreaks[4]
\begin{document}

\title[On functional records and champions]{On functional records and champions}

\author{Cl\'{e}ment Dombry, Michael Falk, Maximilian Zott}
\address{Laboratoire de Math\'{e}matiques, UFR Sciences et Techniques, Universit\'{e} de Franche-Comt\'{e}, France}
\email{clement.dombry@univ-fcomte.fr}
\address{Institute of Mathematics, W\"{u}rzburg, Germany}
\email{michael.falk@uni-wuerzburg.de, maximilian.zott@uni-wuerzburg.de}
\thanks{This work was supported by a research grant (VKR023480) for the second author from VILLUM FONDEN}%



\begin{abstract}
Records among a sequence of iid random variables $X_1,X_2,\dotsc$ on the real line have been investigated extensively over the past decades. A record is defined as a random variable $X_n$ such that $X_n>\max(X_1,\dotsc,X_{n-1})$. Trying to generalize this concept to the case of random vectors, or even stochastic processes with continuous sample paths, the question arises how to define records in higher dimensions. We introduce two different concepts: A \emph{simple record} is meant to be a stochastic process (or a random vector) $\bm X_n$ that is larger than $\bm X_1,\dotsc,\bm X_{n-1}$ in at least one component, whereas a \emph{complete record} has to be larger than its predecessors in all components. The behavior of records is investigated. In particular, the probability that a stochastic process $\bm X_n$ is a record  as $n$ tends to infinity is studied, assuming that the processes are in the max-domain of attraction of a max-stable process. Furthermore, the distribution of $\bm X_n$, given that $\bm X_n$ is a record is derived.
\end{abstract}

\maketitle


\section{Introduction and Preliminaries}

\subsection*{Preliminaries}

Let $S$ be a compact metric space. A \emph{max-stable process} (MSP) $\bm \vartheta=(\vartheta_s)_{s\in S}$ is a stochastic process with non-degenerate univariate margins and sample paths in $C(S):=\{f\in\R^S:~f\text{ continuous}\}$ with the property that there are functions $a_n\in C^+(S):=\{f\in C(S):~f>0\}$, $b_n\in C(S)$, $n\in\N$, such that
\begin{equation}\label{eq:max_stable_process}
\max_{i=1,\dotsc,n}\frac{\bm \vartheta^{(i)}-b_n}{a_n}=_{\mathcal D}\bm\vartheta,
\end{equation}
where $\bm \vartheta ^{(1)},\dotsc,\bm \vartheta^{(n)}$ are independent and identically distributed (iid) copies of $\bm \vartheta$ and $=_{\mathcal D}$ denotes equality in distribution. Note that throughout this paper, each operation such as $\max$, $<$, $\geq$, and so on is meant component\-wise. The class of max-stable distributions coincides with the class of possible limit distributions of linearly standardized
maxima of iid processes, which makes it a class of outstanding interest for extreme value theory. Obviously, the univariate margins of an MSP are max-stable distributions on the real line, and hence belong to the class of either Fr\'{e}chet, Weibull or Gumbel type of distributions. An MSP $\bm\xi=(\xi_s)_{s\in S}$ in $C(S)$ is commonly called \emph{simple max-stable}, if each univariate margin is unit Fr\'{e}chet distributed, i.\,e. $P(\xi_s\leq x)=\exp\left(-x^{-1}\right)$, $x>0$, $s\in S$. Different to that, we call an MSP $\bm\eta=(\eta_s)_{s\in S}$ in $C(S)$ \emph{standard max-stable} (SMSP), if all univariate margins are standard negative exponentially distributed, i.\,e. $P(\eta_s\leq x)=\exp(x)$, $x\leq 0$, $s\in S$. In that case, $\bm\eta=_dn\max_{i=1,\dotsc,n}\bm\eta^{(i)}$ if $\bm\eta^{(1)},\dotsc,\bm\eta^{(n)}$ are iid copies of $\bm\eta$. It can be shown that a process $\bm\vartheta=(\vartheta_s)_{s\in S}$ with continuous sample paths and univariate margins $G_s(x)=P(\vartheta_s\leq x)$, $s\in S$, $x\in\R$, is an MSP iff $(\log(G_s(\vartheta_s)))_{s\in S}$ is an SMSP, see \citet{aulfahozo14}.

Denote by $E(S)$ the set of all real valued bounded functions with only finitely many discontinuities and define $\bar E^-(S):=\{f\in E(S):~f\leq 0\}$. We know from \citet{ginhv90} and \citet{aulfaho11} that a stochastic process $\bm\eta=(\eta_s)_{s\in S}$ is an SMSP iff there exists a stochastic process $\bm Z=(Z_s)_{s\in S}$ with sample paths in $\bar C^+(S):=\{f\in C(S):~f\geq 0\}$ and some constant $c\geq 1$ with $\sup_{s\in s}Z_s=c$ almost surely and $E(Z_s)=1$, $s\in S$, such that
\[
P(\bm\eta\leq f)=\exp\left(-\norm f_D\right):=\exp\left(-E\left(\sup_{s\in S}\abs{f(s)}Z_s\right)\right),\qquad f\in\bar E^-(S).
\]
Note that the condition $P\left(\sup_{s\in S}Z_s=c\right)=1$ can be weakened to $E\left(\sup_{s\in S}Z_s\right)<\infty$, see \citet{dehaf06}.

As a matter of fact, the mapping $\norm\cdot_D$ defines a norm on the linear space $E(S)$. We call it \emph{$D$-norm} with \emph{generator} $\bm Z$. Note that the distribution of a generator is not uniquely determined in general, i.\,e. there might be several different generators of one $D$-norm, but the condition $\sup_{s\in S}Z_s=c$ almost surely yields uniqueness.

The choice of the function space $E(S)$ instead of $C(S)$ may seem uncommon at first, but in fact it allows the smooth incorporation of the finitedimensional theory on max-stable distributions into the functional setup. By a suitable choice of $f\in E(S)$ we obtain for $s_1,\dotsc,s_d\in S$ and $\bm x=(x_1,\dotsc,x_d)\leq\bm0$
\[
P(\bm\eta\leq f)=P\left(\eta_{s_1}\leq x_1,\dotsc,\eta_{s_d}\leq x_d\right)=\exp\left(-E\left(\max_{i=1,\dotsc,d}\abs{x_i}Z_{s_i}\right)\right),
\]
where the right-hand side is the de Haan-Resnick-Pickands representation of a multivariate standard max-stable distribution function (df), cf. \citet{dehar77}, \citet{pick81}. The mapping $\norm{\bm x}_D:=E\left(\max_{i=1,\dotsc,d}\abs{x_i}Z_{s_i}\right)$, $\bm x\in\R^d$, defines a multivariate $D$-norm, see \citet{fahure10} for details.

A stochastic process $\bm\vartheta$ with non-degenerate univariate margins which realizes in $C(S)$ fulfills condition \eqref{eq:max_stable_process} iff there is a continuous process $\bm X$ in the \emph{max-domain of attraction} of $\bm\vartheta$, i.\,e. there are some norming functions $c_n\in C^+(S)$, $d_n\in C(S)$, $n\in\N$, such that
\begin{equation}\label{eq:domain_of_attraction}
\max_{i=1,\dotsc,n}\frac{\bm X^{(i)}-d_n}{c_n}\to_{\mathcal D}\bm\vartheta,
\end{equation}
where $\bm X^{(1)},\bm X^{(2)},\dotsc$ are iid copies of $\bm X$ and $\to_{\mathcal D}$ denotes convergence in distribution, that is, weak convergence of the distributions in $\left(C(S),\norm\cdot_{\infty}\right)$. For details, see e.\,g. \citet[Section 9.2]{dehaf06}. We shortly write $\bm X\in\mathcal D(\bm\vartheta)$ for \eqref{eq:domain_of_attraction}. Relation \eqref{eq:domain_of_attraction} implies in particular that $\bm X$ is in the \emph{functional domain of attraction} of the max-stable process $\bm\vartheta$, that is, for all $f\in E(S)$,
\begin{equation}\label{eq:functional_domain_of_attraction}
P\left(\max_{i=1,\dotsc,n}\frac{\bm X^{(i)}-d_n}{c_n}\leq f\right)=P\left(\frac{\bm X-d_n}{c_n}\leq f\right)^n\to_{n\to\infty}P(\bm\vartheta\leq f).
\end{equation}
For details on the functional domain of attraction, see \citet{aulfaho11} and \citet{aulfahozo14}. Note that in the multivariate context, where random vectors (rv) instead of stochastic processes are considered, \eqref{eq:domain_of_attraction} and \eqref{eq:functional_domain_of_attraction} are equivalent. Recall that a max-stable df on $\R^d$ is always continuous.

\subsection*{Outline and terminology}

In this paper, we deal with different kinds of records of stochastic processes, generally assuming they are in the max-domain of attraction of an MSP and have continuous univariate marginal df. Let $\bm X,\bm X^{(1)},\bm X^{(2)},\dotsc$ be an iid sequence of stochastic processes in $C(S)$. We call $\bm X^{(n)}$ a \emph{simple record}, if $\bm X^{(n)}\not\leq\max_{i=1,\dotsc,n-1}\bm X^{(i)}$, and a \emph{complete record}, if $\bm X^{(n)}>\max_{i=1,\dotsc,n-1}\bm X^{(i)}$. We further define
\begin{align*}
\ubar\pi_n(\bm X)&:=P\left(\bm X^{(n)}\text{ is a simple record }\right),\\
\tbar\pi_n(\bm X)&:=P\left(\bm X^{(n)}\text{ is a complete record }\right).
\end{align*}

By definition, the first observation $\bm X^{(1)}$ is always a record, so we demand $\ubar\pi_1(\bm X)=\bar\pi_1(\bm X)=1$. In the univariate case, where $X,X^{(1)}, X^{(2)},\dotsc$ are simply random variables on the real line, records are much easier to handle, and clearly $\ubar\pi_n(X)=\bar\pi_n(X)=\frac1n$. There are many detailed works on univariate record and record times, see \citet[Sections 6.2 and 6.3]{gal87} and \citet{arnbn98}. Multivariate records have not been discussed that extensively, yet they haven been approached by e.\,g. \citet{golres89}, \citet{golres95} or \citet[Chapter 8]{arnbn98}.

A concept that is closely related to the field of complete records is the so-called \emph{concurrency of extremes}, which is due to \citet{domribsto15}. We say that $\bm X^{(1)},\dotsc,\bm X^{(n)}$ are \emph{sample concurrent}, if
\[
\max_{i=1,\dotsc,n}\bm X^{(i)}=\bm X^{(k)}\text{ for some } k\in\{1,\dotsc,n\}.
\]
In that case, we call $\bm X^{(k)}$ the \emph{champion} among $\bm X^{(1)},\dotsc,\bm X^{(n)}$. We denote the sample concurrence probability by $p_n(\bm X)$ and obtain due to the iid property
\begin{align}
p_n&(\bm X)=P\left(\bigcup_{i=1}^n\left\{\bm X^{(i)}>\max_{1\leq j\neq i\leq n}\bm X^{(j)}\right\}\right)\\\label{eq:complete_record_champion}
&=\sum_{i=1}^nP\left(\bm X^{(i)}>\max_{1\leq j\neq i\leq n}\bm X^{(j)}\right)=nP\left(\bm X^{(n)}>\max_{j=1,\dotsc,n-1}\bm X^{(j)}\right)=n\bar\pi_n(\bm X).\nonumber
\end{align}
Different to records, the concept of multivariate and functional champions is very recent. It has been established in the work of \citet{domribsto15}. In their paper, they derive the limit sample concurrence probability under iid rv $\bm X^{(1)},\dotsc,\bm X^{(n)}$ in $\R^d$. There are also many results on statistical inference in their work.

In Section \ref{sec:champion_probability}, we generalize the limit sample concurrence probability which has been derived in \citet[Theorem 2]{domribsto15} to the case of stochastic processes with continuous sample paths. Further, we compute the distribution of a champion, given that there actually is one. Section \ref{sec:simple_records} deals with simple record times and the distribution of simple records, where all considerations are restricted to the finitedimensional case.

\section{The dual $D$-norm function}\label{sec:dualdnorm}
To begin with, we introduce a mapping which is strongly related to the $D$-norm of an SMSP, and which will be important troughout the whole paper. Let $\norm\cdot_D$ be a $D$-norm generated by $\bm Z=(Z_s)_{s\in S}$ in $\bar C^+(S)$, (recall $E(Z_s)=1$ and $E\left(\sup_{s\in S}Z_s\right)<\infty$). We call the mapping
\[
\dnormf\cdot_D:E(S)\to\R,\qquad f\mapsto\dnormf f_D:=E\left(\inf_{s\in S}\abs{f(s)}Z_s\right),
\]
the \emph{dual $D$-norm function} corresponding to $\norm\cdot_D$. Note that, despite the fact that the generator of $\norm\cdot_D$ is not uniquely determined, the dual $D$-norm function $\dnormf\cdot_D$ does not depend on the choice of the generator of $\norm\cdot_D$. This is a consequence of \citet[Lemma 6]{aulfaho11}. Therefore, the mapping
\[
\norm\cdot_D\to\dnormf\cdot_D
\]
is well-defined, although not one-to-one, since different $D$-norms can lead to the same dual $D$-norm function. One can check that the dual $D$-norm function is always zero if there are at least two independent components $\eta_s,\eta_t$ of the SMSP $\bm\eta$ generated by $\norm\cdot_D$, since Takahashi's theorem (\citet[Theorem 4.4.1]{fahure10}) implies that $\eta_s,\eta_t$ are independent iff $Z_s+Z_t=\max(Z_s,Z_t)$ almost surely, entailing in turn $\min(Z_s,Z_t)=0$ almost surely.

Throughout this paper, the following result on SMSP will be crucial. By a \emph{copula process}, we understand a stochastic process with continuous sample paths such that the univariate margins are uniformly distributed on $(0,1)$. We are interested in copula processes $\bm U$ that are in the max-domain of attraction of an SMSP $\bm\eta$, i.\,e.
\begin{equation}\label{eq:doa_copula_process}
n\left(\max_{i=1,\dotsc,n}\bm{U}^{(i)}-1\right)\to_{\mathcal D}\bm\eta,
\end{equation}
where $\bm U^{(1)},\bm U^{(1)},\dotsc$ are iid copies of $\bm U$.

\begin{prop}\label{prop:doa_copula_process}
Let $\bm U$ be a copula process with $\bm U\in\mathcal D(\bm\eta)$ (i.\,e. \eqref{eq:doa_copula_process} holds), where $\bm\eta=(\eta_s)_{s\in S}$ is an SMSP. Let $\norm\cdot_D$ be the $D$-norm corresponding to $\bm\eta$ and $\bm Z$ be a generator of $\norm\cdot_D$. Then
\begin{equation}\label{eq:fdoa_copula_process_expansion}
n\left(1-P\left(n(\bm U-1)\leq f\right)\right)\to_{n\to\infty}E\left(\sup_{s\in S}\abs{f(s)}Z_s\right)=\norm f_D,\qquad f\in \bar E^-(S),
\end{equation}
and
\begin{equation}\label{eq:survival_function_expansion}
nP\left(n(\bm U-1)>f\right)\to_{n\to\infty}E\left(\inf_{s\in S}\abs{f(s)}Z_s\right)=\dnormf{f}_D,\qquad f\in \bar E^-(S).
\end{equation}
\end{prop}

\begin{rem}
\upshape We call a stochastic process $\bm V$ with sample paths in $\bar C^-(S):=\{f\in C(S):~f\leq 0\}$ a \emph{standard generalized Pareto process} (standard GPP), if there is a $D$-norm $\norm\cdot_D$ on $E(S)$ generated by an almost surely bounded generator and some $c>0$ such that
\[
P(\bm V\leq f)=1-\norm f_D
\]
for all $f\in \bar E^-(S)$ with $\norm f_{\infty}\leq c$. It can easily be shown that the survival function of $\bm V$ is given by
\[
P(\bm V>tf)=t\dnormf f_D
\]
for $t>0$ close enough to zero. Hence, condition \eqref{eq:fdoa_copula_process_expansion} and \eqref{eq:survival_function_expansion} mean that the upper tail of the distribution of the copula process $\bm U$ is close to that of the shifted standard GPP $\bm V+1$. For details on GPP, see e.\,g. \citet{buihz08}, \citet{aulfa11}, \citet{aulfa12}, and \citet{ferrdh12}.
\end{rem}

\begin{proof}[Proof of Proposition \ref{prop:doa_copula_process}]
Condition \eqref{eq:doa_copula_process} implies that $\bm U$ is in the functional domain of attraction of $\bm\eta$, see \citet[Proposition 5]{aulfaho11}, i.\,e.
\begin{equation*}
P\left(n(\bm U-1)\leq f\right)^n\to_{n\to\infty}P(\bm\eta\leq f)=\exp\left(-\norm f_D\right),\qquad f\in \bar E^-(S).
\end{equation*}
Now \eqref{eq:fdoa_copula_process_expansion} follows from \citet[Proposition 8]{aulfaho11}. Next we verify \eqref{eq:survival_function_expansion}. Choose a generator $\bm Z$ of $\norm\cdot_D$ with $P(\sup_{s\in S}Z_s=c)=1$ for some $c\geq 1$. Define a measure $\rho$ on the unit sphere $\bar C_1^+(S):=\{g\in \bar C^+(S):~\norm g_{\infty}=1\}$ by
\[
\rho(A):=cP(\bm Z/c\in A),\qquad A\subset \bar C_1^+(S)\text{ Borel},
\]
which is the well-known \emph{angular measure}, see e.\,g. \citet[Section 9.4]{dehaf06}. By transforming to polar coordinates, we identify $\bar C^+(S)$ with the product space $\bar C^+_1(S)\times(0,\infty)$. For the technical details of this transformation, see \citet[Section 9.3]{dehaf06}. On this product space, we define a product measure via $d\nu=d\rho\times dr/r^2$. The measure $\nu$ is well-known as the \emph{exponent measure} in the literature, see again \citet[Section 9.3]{dehaf06} among many others. Now having in mind that $\bm U<1$ a.\,s. (\citet[Corollary 3.15]{hofm13}) and $\bm\eta<0$ a.\,s. (\citet[Lemma 1]{aulfaho11}), it is easy to see that \eqref{eq:doa_copula_process} is equivalent with
\[
\frac1n\max_{i=1,\dotsc,n}\frac1{1-\bm U^{(i)}}\to_{\mathcal D}-\frac1{\bm\eta},
\]
where $-1/{\bm\eta}$ is a \emph{simple} MSP. Therefore, we have
\[
\nu_n(A):=nP\left((n(1-\bm U))^{-1}\in A\right)\to_{n\to\infty}\nu(A)
\]
for all Borel sets $A\subset\bar C^+(S)$ with $\nu(\partial A)=0$ and $\inf\{\norm f_{\infty}:~f\in A\}>0$, see \citet[Theorem 9.3.1]{dehaf06}. Define for $h\in E(S)$ the set $A_h:=\{g\in C(S):g>h\}$. Now, for all $f\in\bar E^-(S)$ with $f<0$,
\begin{align*}
nP(n(\bm U-1)>f)&=\nu_n\left(A_{-1/f}\right)\\
&\to_{n\to\infty}\nu(A_{-1/f})\\
&=\nu\left(\left\{(g,r)\in\bar C^+_1(S)\times(0,\infty):~rg>1/\abs f \right\}\right)\\
&=\int_{\bar C^+_1(S)}\int_{\left(\inf_{s\in S}\abs{f(s)}g(s)\right)^{-1}}^{\infty}r^{-2}~\d r~\rho(\d g)\\
&=E\left(\inf_{s\in S}\abs{f(s)}Z_s\right).
\end{align*}

\end{proof}

\begin{rem}\label{rem:multivariate_doa}\upshape
Clearly, the dual $D$-norm function can also be defined for multivariate $D$-norms. Given a multivariate $D$-norm $\norm{\bm x}_D=E\left(\max_{j=1,\dotsc,d}\abs{x_j}Z_j\right)$, $\bm x\in\R^d$, we write
\[
\dnormf{\bm x}_D:=E\left(\min_{j=1,\dotsc,d}\abs{x_j}Z_j\right),\qquad \bm x\in\R^d.
\]
A simple connection between the functions $\norm\cdot_D$ and $\dnormf\cdot_D$ is now given by the general equation
\begin{equation}\label{eq:max=min}
\min(a_1,\dots,a_d)= \sum_{\emptyset\not=T\subset\set{1,\dots,d}}(-1)^{\abs T-1} \max\set{a_j,\,j\in T},
\end{equation}
which is true for arbitrary numbers $a_1,\dotsc,a_d\in\R$. Applying the inclusion-exclusion principle and including \eqref{eq:max=min}, the multivariate version of \eqref{eq:survival_function_expansion} directly follows from that of \eqref{eq:fdoa_copula_process_expansion}. While \eqref{eq:fdoa_copula_process_expansion} traces back to \citet{deheu84} and \citet{gal87}, the multivariate dual $D$-norm function was established by \citet{schmst06}, see also \citet{dehanepe08}. In their work, ${\displaystyle\dnormf\cdot_D}$ is called \emph{tail copula}. However, they do not provide an explicite formula for the tail copula.
\end{rem}

\begin{exam}[Independence and perfect dependence]\upshape
We have that
\[
\dnormf\cdot_1=0
\]
is the least dual $D$-norm function, corresponding to the case of independent univariate margins, where $\norm\cdot_D=\norm\cdot_1$, and
\[
\dnormf{\bfx}_\infty=\min_{1\le j\le d}\abs{x_j},\qquad \bfx\in\R^d,
\]
is the largest dual $D$-norm function, corresponding to the perfect dependence case, where $\norm\cdot_D=\norm\cdot_\infty$. Hence, we have for an arbitrary dual $D$-norm function the bounds
\[
0=\dnormf\cdot_1\le \dnormf\cdot_D\le \dnormf\cdot_\infty.
\]
\end{exam}

For the next examples, the following abbreviation is useful. We define for $\bm x\in\R^d$ and a nonempty subset $T\subset\{1,\dotsc,d\}$
\[
\bm x_T:=(x_i,~i\in T)\in\R^{\abs T}.
\]

\begin{exam}[Fr\' echet model]\label{exam:frechet}\upshape
It is well-known that a $D$-norm is given by the $l_\lambda$-norm
\begin{equation}\label{eq:lp_norm}
\norm{\bm x}_\lambda:=\left(\sum_{i=1}^d\abs{x_i}^\lambda\right)^{1/\lambda},\qquad \bm x\in\R^d,~\lambda\in(1,\infty),
\end{equation}
usually referred to as the \emph{logistic model} in the literature. Therefore, we obtain by \eqref{eq:max=min}
\[
\dnormf{\bm x}_\lambda=\sum_{\emptyset\neq T\subset\{1,\dotsc,d\}}(-1)^{\abs T-1}\norm{\bm x_T}_\lambda,\qquad \bm x\in\R^d,~\lambda\in(1,\infty).
\]
A generator $\bm Z=(Z_1,\dotsc,Z_d)$ of $\norm\cdot_\lambda$ can easily be found: Put $Z_i:=\tilde Z_i/\Gamma\left(1-1/\lambda\right)$, $i=1,\dotsc,d$, where $\tilde Z_1,\dotsc,\tilde Z_d$ are iid Fr\'{e}chet distributed with parameter $\lambda$, and $\Gamma$ denotes the gamma function.
\end{exam}

\begin{exam}[Weibull model]\label{exam:weibull}\upshape
We can define a generator $\bm Z=(Z_1,\dotsc,Z_d)$ by taking independent Weibull distributed random variables $\tilde Z_1,\dotsc,\tilde Z_d$, i.\,e. $P(\tilde Z_1>t)=\exp(-t^\alpha)$, $t>0$, $\alpha>0$, and putting $Z_i:=\tilde Z_i/\Gamma(1+1/\alpha)$. It is easy to show that the corresponding dual $D$-norm function is for $\bm x\in\R^d$, $x_i\neq0$, $i=1,\dotsc,d$, given by
\begin{equation}\label{eq:dualdnorm_weibull}
\dnormf{\bm x}_{W_\alpha}=\left(\norm{1/\bm x}_\alpha\right)^{-1},\qquad \alpha>0.
\end{equation}
Hence, by \eqref{eq:max=min}, the attendant $D$-norm is for such $\bm x$
\[
\norm{\bm x}_{W_\alpha}=\sum_{\emptyset\neq T\subset\{1,\dotsc,d\}}(-1)^{\abs T-1}\left(\norm{1/\bm x_T}_\alpha\right)^{-1},\qquad\alpha>0.
\]
Note that $\norm\cdot_\alpha$ is defined as in \eqref{eq:lp_norm}, even though it does not define a norm in the case $\alpha<1$.
\end{exam}

\begin{exam}[Bernoulli model]\label{exam:bernoulli}\upshape
A simple example of a discrete generator is induced by independent Bernoulli-$\beta$ random variables $\tilde Z_i$, $i=1,\dotsc,d$, $\beta\in(0,1]$, and putting $Z_i:=\tilde Z_i/\beta$, $i=1,\dotsc,d$. The $D$-norm and the dual $D$-norm function are easily derived. We have
\[
\norm{\bm x}_{B_\beta}=\sum_{\emptyset\neq T\subset\{1,\dotsc,d\}}\beta^{\abs T-1}(1-\beta)^{d-\abs T}\norm{\bm x_T}_\infty,\qquad\bm x\in\R^d,~\beta\in(0,1].
\]
Note that $\norm\cdot_{B_1}=\norm\cdot_\infty$ and $\norm\cdot_{B_\beta}\to\norm\cdot_1$ as $\beta\to0$. Analogously,
\[
\dnormf{\bm x}_{B_\beta}=\sum_{\emptyset\neq T\subset\{1,\dotsc,d\}}\beta^{\abs T-1}(1-\beta)^{d-\abs T}\dnormf{\bm x_T}_\infty,\qquad\bm x\in\R^d,~\beta\in(0,1].
\]
\end{exam}

\section{The functional extremal concurrence probability}\label{sec:champion_probability}

The aim of this section is to investigate the limit behaviour of the sample concurrence probability. In \citet{domribsto15}, it is shown that the sample concurrence probability $p_n(\bm X)$ of a rv $\bm X$ converges, provided that $\bm X$ has continuous margins and lies in the max-domain of attraction of a max-stable rv. We generalize this assertion to the functional setup, having in mind that we can formulate every result in the multivariate context analogously.

\begin{theorem}\label{thm:champions_copula_processes}
Let $\bfU^{(1)},\bfU^{(2)},\ldots$ be independent copies of a copula process $\bfU$, satisfying $\bm U\in\mathcal D(\bm\eta)$, where $\bm\eta$ is an SMSP with corresponding $D$-norm $\norm\cdot_D$. Then
\begin{equation*}
p_n(\bm U)=n\bar\pi_n(\bm U)\to_{n\to\infty} E\left(\dnormf{\bm\eta}_D\right),
\end{equation*}
where $\dnormf\cdot_D$ is the dual $D$-norm function corresponding to $\norm\cdot_D$.
\end{theorem}

We call $E\left(\dnormf{\bm\eta}_D\right)$ the \emph{extremal concurrence probability} corresponding to $\norm\cdot_D$, in accordance with the terminology in \citet{domribsto15}. As they have shown, the extremal concurrence probability has the following interpretation. It is well-known (cf. \citet[Corollary 9.4.2]{dehaf06}), that the simple max-stable process $\bm\xi=-1/\bm\eta$ has the representation
\begin{equation}\label{eq:ppp_representation}
\bm\xi=_{\mathcal D}\sup_{k\in\N}\bm\vartheta_k,
\end{equation}
where $(\bm\vartheta_k)_{k\in\N}$ are the points of a Poisson point process on $(0,\infty)\times\bar C_1^+(S)$ with a certain intensity measure. The extremal concurrence probability is now precisely the probability that only one function $\bm\vartheta_k$ contributes to the supremum in \eqref{eq:ppp_representation}, see \citet[Theorem 1]{domribsto15}.

Note that one has to distinguish between $E\left(\dnormf{\bm\eta}_D\right)$ and $E\left(\inf_{s\in S}\abs{\eta_s}Z_s\right)$ in general. However, if $\bm\eta$ and $\bm Z$ are independent, both terms coincide, cf. Lemma \ref{lemma:other_formula_for_champion_probability}.

\begin{proof}[Proof of Theorem \ref{thm:champions_copula_processes}]
Denote by $P*\xi$ the distribution of a random variable $\xi$. Let $\bm\eta$ be an SMSP with $D$-norm $\norm\cdot_D$ and put $\bm M^{(n)}:=n\max_{i=1,\dotsc,n-1}\left(\bm U^{(i)}-1\right)\to_{\mathcal D}\bm\eta$ due to \eqref{eq:doa_copula_process}. Conditioning on $\bm M^{(n)}=f$ yields
\begin{align*}
n\bar\pi_n(\bm U)&=\int_{\bar C^-(S)}nP\left(n(\bm U-1)>f\right)~\left(P*\bm M^{(n)}\right)(\d f)\\
&=:\int_{\bar C^-(S)}G_n(f)~\left(P*\bm M^{(n)}\right)(\d f)
\end{align*}
since $\bm M^{(n)}$ and $\bm U$ are independent. Setting $X_n:=G_n\circ\bm M^{(n)}$, we need to show
\[
n\bar\pi_n(\bm U)=E(X_n)\to_{n\to\infty}=E\left(\dnormf{\bm\eta}_D\right).\\
\]
It is enough to verify (\citet[p. 32]{billi68}):
\begin{enumerate}[(i)]
\item $X_n\to_{\mathcal D} \dnormf{\bm\eta}_D$.
\item There is $\eps>0$ with $\sup_{n\in\N}E\left(\abs{X_n}^{1+\eps}\right)<\infty$.
\end{enumerate}
Note that (ii) implies the \emph{uniform integrability} of the sequence $(X_n)_{n\in\N}$.

We first show (i). Obviously, $G_n(f)\to \dnormf f_D$ due to \eqref{eq:survival_function_expansion}. Standard arguments such as the monotone convergence theorem yield $G_n(f_n)\to \dnormf f_D$ if $f_n,f \in \bar C^-(S)$ with $\norm{f_n-f}_{\infty}\to 0$. Now noticing that $\bm M^{(n)}\to_{\mathcal D}\bm\eta$, the assertion is immediate from the extended continuous mapping theorem, see cf. \citet[Theorem 5.5]{billi68}.

Now we proof (ii). Elementary calculations show that for all $n\geq 2$
\begin{align*}
E\left(X_n^2\right)&=\int_{\bar C^-(S)}n^2P\left(n(\bm U-1)>f\right)^2~\left(P*\bm M^{(n)}\right)(\d f)\\
&\leq \int_{\bar C^-(S)}n^2P\left(n\left(U_s-1\right)>f(s)\right)^2~\left(P*\bm M^{(n)}\right)(\d f)\\
&=E\left(\left(M_s^{(n)}\right)^2\right)=\frac{2n}{n+1}\leq 2.
\end{align*}
\end{proof}

\begin{cor}\label{cor:number_of_complete_records}
Denote by $M(n):=\sum_{i=1}^n 1_{\left\{\bfX^{(i)}>\max_{1\le j<i}\bfX^{(j)}\right\}}$ the number of complete records among $\bfX^{(1)},\ldots,\bfX^{(n)}$. Then
\[
\frac{E(M(n))}{\log(n)}\to_{n\to\infty} E\left(\dnormf{\bm\eta}_D\right).
\]
\end{cor}

\begin{proof}
The assertion follows from Theorem \ref{thm:champions_copula_processes} and the fact that  $\left(\sum_{i=1}^n\frac{a_i}{i}\right)/\log(n)\linebreak\to_{n\to\infty}a$, if $(a_n)_{n\in\N}$ is some real-valued sequence with $a_n\to_{n\to\infty}a$.
\end{proof}

The following lemma provides an alternative representation for the extremal concurrence probability. Denote by $1_A$ the indicator function of some set $A$, i.\,e. $1_A(\omega)=1$, if $\omega\in A$, and $1_A(\omega)=0$, else.

\begin{lemma}\label{lemma:other_formula_for_champion_probability}
Let $\bm\eta=(\eta_s)_{s\in S}$ be an SMSP in $\bC$ with $D$-norm $\norm\cdot_D$ and generator $\bm Z=(Z_s)_{s\in S}$, and $f\in\bE$. Then
\begin{enumerate}[(i)]
\item
\begin{equation*}
E\left(\dnormf{\bm\eta}_D\right)=E\left(\norm{1/\bm Z}_D^{-1}1_{\{\bm Z>0\}}\right).
\end{equation*}
\item
\begin{align*}
&E\left(\dnormf{\max(\bm\eta,f)}_D\right)=\\
&=E\left(\left(\norm{1/\bm Z}_D\right)^{-1}\left(1-\exp\left(\norm{1/\bm Z}_D\sup_{s\in S}(f(s)Z_s)\right)\right)1_{\{\bm Z>0\}}\right).
\end{align*}
\end{enumerate}
\end{lemma}

\begin{proof}
Without loss of generality, choose a generator $\bm Z$ of $\norm\cdot_D$ which is independent of $\bm\eta$. Then
\[
E\left(\inf_{s\in S}\abs{\eta_s}Z_s\right)=\int_{\bar C^-(S)}{\displaystyle\dnormf f_D}~(P*\bm\eta)(df)=E\left(\dnormf{\bm\eta}_D\right).
\]
Suppose $P(\bm Z>0)=1$ for ease of notation. Fubini's theorem and the fact that $\bm\eta$ and $\bm Z$ are independent, entail
\begin{align*}
E\left(\inf_{s\in S}\left(\abs{\eta_s}Z_s\right)\right)&=\int_0^{\infty}P\left(\inf_{s\in S}\left(\abs{\eta_s}Z_s\right)>t\right)~\d t\\
&=\int_0^{\infty}P\left(\eta_s<-t/Z_s,~s\in S\right)~\d t\\
&=E\left(\int_0^{\infty}\exp\left(-t\norm{1/\bm Z}_D\right)~\d t\right)\\
&=E\left(\left(\norm{1/\bm Z}_D\right)^{-1}\int_0^{\infty}\exp\left(-t\right)~\d t\right),
\end{align*}
which is (i). Assertion (ii) can be shown by similar arguments.
\end{proof}

\begin{exam}[Independence and perfect dependence]\upshape
A generator of the special $D$-norm $\norm\cdot_D=\norm\cdot_\infty$, which characterizes the complete dependence of the univariate margins of $\bm\eta$, is obviously given by the constant $\bm Z\equiv1$. In that case, Theorem \ref{thm:champions_copula_processes} shows that the extremal concurrence probability is one, i.\,e. $p_n(\bm U)=n\bar\pi_n(\bm U)\to_{n\to\infty}1$. This is not at all surprising: in the univariate context, where $X^{(1)},\dotsc,X^{(n)}$ are random variables on the real line, there clearly exists a champion with probability one - it is the maximum of $X^{(1)},\dotsc,X^{(n)}$.

In contrast to that, we have
\begin{equation}\label{eq:concurrence_probability_zero}
E\left(\left(\norm{1/\bm Z}_D\right)^{-1}1_{\{\bm Z>0\}}\right)=0\iff\inf_{s\in S}Z_s=0\text{ a.\,s.}
\end{equation}
In particular, this is the case when at least two components $\eta_s$, $\eta_t$, $s\neq t$, are independent, see the argument in Section \ref{sec:dualdnorm}.
\end{exam}

\begin{exam}[Bernoulli model]\upshape
Consider a standard max-stable rv $\bm\eta\in\R^d$ with corresponding $D$-norm $\norm\cdot_{B_\beta}$, $\beta\in(0,1]$, known from Example \ref{exam:bernoulli}. It is easy to see that
\[
\norm{\bm 1}_{B_\beta}=\frac{1-(1-\beta)^d}{\beta}.
\]
From the general equality
\[
E\left(\dnormf{\tilde{\bm\eta}}_\infty\right)=\frac1{\norm{\bm 1}_D},
\]
where $\tilde{\bm\eta}$ is some standard max-stable rv with $D$-norm $\norm\cdot_D$, we conclude
\begin{align*}
E\left(\dnormf{\bm\eta}_{B_\beta}\right)&=\sum_{\emptyset\neq T\subset\{1,\dotsc,d\}}\beta^{\abs T-1}(1-\beta)^{d-\abs T}E\left(\dnormf{\bm\eta_T}_\infty\right)\\
&=\sum_{k=1}^d\binom dk\beta^k\frac{(1-\beta)^{d-k}}{1-(1-\beta)^k}
\end{align*}
\end{exam}

For another example, namely the logistic model, we refer to Example \ref{exam:frechet_records}.

\begin{rem}\label{rem:arbitrary_margins}
\upshape
\begin{enumerate}[(i)]
\item Theorem \ref{thm:champions_copula_processes} implies that the extremal concurrence probability, just like the dual $D$-norm function, does not depend on the choice of $\bm Z$, but only on $\norm\cdot_D$.
\item In the preceding theorem, we can replace $\bm U,\bm U^{(1)},\bm U^{(2)},\dotsc$ by a sequence of iid stochastic processes $\bm X,\bm X^{(1)},\bm X^{(2)},\dotsc$ whose univariate marginal df $F_s(x)=P(X_s\leq x)$, $s\in S$, are continuous and strictly monontonically increasing on their support. The conditions \eqref{eq:fdoa_copula_process_expansion} and \eqref{eq:survival_function_expansion} will then have to apply to the copula process $(F_s(X_s))_{s\in S}$. In that case,
\[
nP\left(\bm X>\max_{i=1,\dotsc,n-1}\bm X^{(i)}\right)\to_{n\to\infty} E\left(\left(\norm{1/\bm Z}_D\right)^{-1}1_{\{\bm Z>0\}}\right),
\]
where $\bm Z$ is a generator of the $D$-norm corresponding to the copula expansion of $(F_s(X_s))_{s\in S}$. Hence, the probability that there is a champion among $\bfX^{(1)},\ldots,\bfX^{(n)}$ does not depend on the univariate margins, but rather on the copula process of $\bm X$.

\end{enumerate}
\end{rem}

The above remark shows that we do not have to limit our considerations to copula processes. If, for instance, $\bm X$ is an MSP itself with univariate marginal distributions $G_s$, $s\in S$, then $\bm\eta:=(\log\left(G_s(X_s)\right))_{s\in S}$ is an SMSP. Applying the max-stability of $\bm\eta$, we obtain
\begin{align*}
\bar\pi_n(\bm X)=\bar\pi_n(\bm\eta)&=P\left(\bm \eta>\max_{i=1,\dotsc,n-1}\bm \eta^{(i)}\right)\\
&=P\left((n-1)\bm\eta>\bm\eta^{(1)}\right)\\
&=\int_{\bC}P\left((n-1) f>\bm\eta^{(1)}\right)~(P\ast \bm\eta)(\d f)\\
&=\int_{\bC}\exp\left(-(n-1)\norm f_D\right)~(P\ast \bm\eta)(\d f)\\
&=E\left(\exp\left(-(n-1)\norm{\bm\eta}_D\right)\right),
\end{align*}
where $\norm\cdot_D$ is the $D$-norm corresponding to $\bm\eta$, and $\bm\eta^{(1)},\bm\eta^{(2)},\dotsc$ are iid copies of $\bm\eta$.

Having established the functional extremal concurrence probability, we can now derive the limit survival function of a complete record. We will have to restrict to the case where $P(\bm Z>0)>0$, which is equivalent to the fact that the extremal concurrence probability is positive, cf. \eqref{eq:concurrence_probability_zero}.

Just like before, we consider the copula process case first.

\begin{prop}\label{prop:survival_function_champion}
In addition to the assumptions of Theorem \ref{thm:champions_copula_processes}, suppose that the generator fulfills $P(\bm Z>0)>0$. Then, for $f\in\bE$,
\begin{align*}
&P\left(n\left(\bm U^{(n)}-\bm 1\right)>f\Big|\bm U^{(n)}\textnormal{ is a complete record} \right)\\
&\hspace*{2cm}=:\bar H_n(f)\to_{n\to\infty} \bar H_D(f):=\frac{\displaystyle E\left(\dnormf{\max(\bm\eta,f)}_D\right)}{E\left(\dnormf{\bm\eta}_D\right)},
\end{align*}
where $\bfeta=(\eta_s)_{s\in S}$ is an SMSP with corresponding $D$-norm $\norm\cdot_D$.
\end{prop}

Note that we avoid division by zero in the preceding formula since we assume $P(\bm Z>0)>0$.

\begin{proof}[Proof of Proposition \ref{prop:survival_function_champion}]
For the ease of notation, we write $\pi_n$ instead of $\bar\pi_n(\bm U)$. We have
\[
\bar H_n(f)=\frac{\Pi_n(f)}{\pi_n}:=\frac{\displaystyle P\left(n(\bm U-1)>f,\bm U>\max_{i=1,\dotsc,n-1}\bm U^{(i)}\right)}{\displaystyle P\left(\bm U>\max_{i=1,\dotsc,n-1}\bm U^{(i)}\right)}.
\]
By Theorem \ref{thm:champions_copula_processes}, it remains to show that for each $f\in\bE$
\begin{align*}
n\Pi_n(f)&=nP\left(n(\bm U- 1)>\max\left(f,\bm M^{(n)}\right)\right)\to_{n\to\infty}E\left(\dnormf{\max(\bm\eta,f)}_D\right),
\end{align*}
where $\bm M^{(n)}:=n\max_{i=1,\dotsc,n-1}\left(\bm U^{(i)}-1\right)$. This can be done by repeating the arguments of the proof of Theorem \ref{thm:champions_copula_processes}.
\end{proof}

Note that another representation of $\bar H_D(f)$ is given by
\begin{equation}\label{eq:survival_function_champion}
\bar H_D(f)=1-\frac{\displaystyle E\left(\left(\norm{1/\bm Z}_D\right)^{-1}\exp\left(\norm{1/\bm Z}_D\sup_{s\in S}(f(s)Z_s)\right)\cdot  1_{\{\bm Z> 0\}}\right)}{E\left(\left(\norm{1/\bm Z}_D\right)^{-1}\cdot  1_{\{\bm Z> 0\}}\right)},
\end{equation}
where $\bm Z$ is a generator of $\norm\cdot_D$. This is due to Lemma \ref{lemma:other_formula_for_champion_probability}.

\begin{exam}
\upshape
For the Marshall-Olkin $D$-norm
\[
\norm{\bm x}_{M_\gamma}:=\gamma\norm{\bm x}_{\infty}+(1-\gamma)\norm{\bm x}_{1},\qquad \bm x\in\R^d,~\gamma\in(0,1),
\]
we obtain with $\bm 1:=(1,\dotsc,1)\in\R^d$
\[
\bar H_{\gamma}(\bm x)=1-\exp\left(\norm{\bm 1}_{M_\gamma}\max_{i=1,\dotsc,d}x_i\right),\qquad \bm x\leq0,
\]
which is the survival function of the max-stable rv $(\eta,\dotsc,\eta)/\norm{\bm 1}_{\gamma}$, where $\eta$ is standard negative exponentially distributed and $\norm{\bm 1}_{M_\gamma}=\gamma+d(1-\gamma)$. Note that this rv has complete dependent and identically distributed univariate margins.
\end{exam}

\begin{proof}
A generator of the Marshall-Olkin $D$-norm $\norm{\cdot}_{M_\gamma}$ is given by
\[
\bm Z:=\xi(1,\dotsc,1)+(1-\xi)\bm Z^\ast,
\]
where $\xi$ is a rv with $P(\xi=1)=\gamma=1-P(\xi=0)$, and $\xi$ is independent of $\bm Z^\ast$ which is a random permutation of the vector $(d,0,\dotsc,0)$ with equal probability $1/d$. Obviously, $P(\bm Z> 0,\xi=0)=0$. On the other hand, $\xi=1$ implies $\bm Z=1$. Thus, we obtain by \eqref{eq:survival_function_champion} for all $\bm x\leq\bm 0$
\begin{align*}
\bar H_{\gamma}(\bm x)&=1-\frac{\displaystyle E\left(\left(\norm{1/\bm Z}_{M_\gamma}\right)^{-1}\exp\left(\norm{1/\bm Z}_{M_\gamma}\max_{i=1,\dotsc,d}(x_iZ_i)\right)\cdot  1_{\{\bm Z> 0,\xi=1\}}\right)}{E\left(\left(\norm{1/\bm Z}_{M_\gamma}\right)^{-1}\cdot  1_{\{\bm Z>0,\xi=1\}}\right)}\\
&=1-\exp\left(\norm{\bm 1}_{M_\gamma}\max_{i=1,\dotsc,d}x_i\right).
\end{align*}
\end{proof}

In order to generalize Proposition \ref{prop:survival_function_champion} to stochastic processes in $C(S)$ with arbitrary margins, the following lemma is needed.

\begin{lemma}\label{lemma:survival_function_champion_uniformly}
Let $f_n$, $n\in\N$, be a sequence of functions in $\bE$ converging uniformly to $f\in\bE$. Then, under the conditions and notation of Proposition \ref{prop:survival_function_champion},
\[
\bar H_n(f_n)=\frac{\Pi_n(f_n)}{\pi_n}\to_{n\to\infty}\bar H_D(f).
\]
\end{lemma}

\begin{proof}
Let $\eps>0$. Due to the uniform convergence of $f_n$, there exists $N\in\N$ such that $f-\eps\leq f_n\leq f+\eps$ for $n\geq N$. Assume without loss of generality $f+\eps<0$, otherwise consider $\min(f+\eps,0)$.  Clearly, for such $n$,
\[
\Pi_n(f+\eps)\leq\Pi_n(f_n)\leq\Pi_n(f-\eps).
\]
Now with $n\to\infty$, Proposition \ref{prop:survival_function_champion} shows
\[
E\left(\inf_{s\in S}\abs{\max\left(\eta_s,f(s)-\eps\right)}Z_s\right)\leq \lim_{n\to\infty}\Pi_n(f_n)\leq E\left(\inf_{s\in S}\abs{\max\left(\eta_s,f(s)+\eps\right)}Z_s\right).
\]
Now check
\[
\inf_{s\in S}\abs{\max\left(\eta_s,f(s)\pm\eps\right)}Z_s\leq-\eta_{s_0}Z_{s_0},\qquad s_0\in S,
\]
and let $\eps\downarrow0$. The assertion now follows from the dominated convergence theorem.
\end{proof}

We are now ready to generalize Proposition \ref{prop:survival_function_champion} to stochastic processes in $C(S)$ with arbitrary univariate margins. Let $\bm X=(X_s)_{s\in S}$ be a process in $C(S)$ whose univariate marginal dfs $F_s(x)=P(X_s\leq x)$, $x\in\R$, $s\in S$, are continuous and strictly monotonically increasing on their support. Let $\bm\vartheta$ be an MSP with univariate marginal dfs $G_s(x)=P(\vartheta_s\leq x)$, $x\in\R$, $s\in S$. We conclude from \citet[Theorem 2.8]{dehal01} that $\bm X$ is in the max-domain of attraction of $\bm\vartheta$ (in the sense of \eqref{eq:domain_of_attraction}) if and only if the \emph{copula process corresponding to $\bm X$}, namely
\[
\bm U=(U_s)_{s\in S}:=(F_s(X_s))_{s\in S},
\]
is in the max-domain of attraction of the SMSP $\bm\eta=(\eta_s)_{s\in S}=:\left(\log(G_s(\vartheta_s))\right)_{s\in S}$ and the univariate margins fulfill
\begin{equation}\label{eq:convergence_univariate_margins}
F_s(c_n(s)x+d_n(s))^n\to_{n\to\infty}G_s(x),\qquad x\in\R,
\end{equation}
uniformly for $s\in S$ and locally uniformly for $x\in\R$, where $c_n\in C^+(S)$, $d_n\in C(S)$, $n\in\N$, are the norming functions from \eqref{eq:domain_of_attraction}.


\begin{cor}\label{cor:survival_function_champion_arbitrary_margins}
Let $\bm\vartheta$ be an MSP with univariate marginal dfs $G_s$, $s\in S$, and $\bfX^{(1)},\bfX^{(2)},\ldots$ be independent copies of a process $\bfX\in\mathcal D(\bm\vartheta)$ in $C(S)$. Let $c_n\in C^+(S)$, $d_n\in C(S)$, $n\in\N$, be the norming functions from \eqref{eq:domain_of_attraction}, and suppose the univariate margins of $\bm X$ satisfy \eqref{eq:convergence_univariate_margins}.  Put
\[
\bm U=(U_s)_{s\in S}:=(F_s(X_s))_{s\in S},\quad \bm\eta=(\eta_s)_{s\in S}=:\left(\log(G_s(\vartheta_s))\right)_{s\in S},
\]
and let $\norm\cdot_D$ be the $D$-norm of $\bm\eta$. Choose a generator $\bm Z=(Z_s)_{s\in S}$ of $\norm\cdot_D$ and suppose that $P(\bm Z>0)>0$. Then, for $f\in E(S)$ with $\inf_{s\in S}G_s(f(s))>0$,
\[
P\left(\frac{\bm X^{(n)}-d_n}{c_n}>f\Big|\bm X^{(n)}\textnormal{ is a complete record}\right)\to_{n\to\infty} \bar H_D(\psi(f))
\]
where $\psi(f)(s)=\log\left(G_s(f(s))\right)$, $s\in S$.
\end{cor}

\begin{proof}
Denote by $\bm U^{(n)}$ the copula process corresponding to $\bm X^{(n)}$, $n\in\N$. Taking logarithms, \eqref{eq:convergence_univariate_margins} becomes
\begin{equation}\label{eq:uniform_convergence_univariate_margins}
\sup_{s\in S}\abs{n\left(F_s(c_n(s)x+d_n(s))-1\right)-\log\left(G_s(x)\right)}\to_{n\to\infty}0.
\end{equation}
It can be shown by elementary arguments that \eqref{eq:uniform_convergence_univariate_margins} is equivalent to
\begin{align*}
&\sup_{s\in S}\abs{\psi_n(f(s))-\psi(f(s))}:=\\
&\quad\sup_{s\in S}\abs{n\left(F_s(c_n(s)f(s)+d_n(s))-1\right)-\log\left(G_s(f(s))\right)}\to_{n\to\infty}0
\end{align*}
for each $f\in E(S)$ with $\inf_{s\in S}G_s(f(s))>0$. Hence, Lemma \ref{lemma:survival_function_champion_uniformly} and the strict monotonicity of $F_s$ entail
\begin{align*}
&nP\left(\frac{\bm X-d_n}{c_n}>f,\bm X>\max_{i=1,\dotsc,n-1}\bm X^{(i)}\right)\\
&=nP\left(n\left(U_s-1\right)>\psi_n(f(s)),~s\in S,~\bm U>\max_{i=1,\dotsc,n-1}\bm U^{(i)}\right)\\
&\to_{n\to\infty} E\left(\dnormf{\max(\bm\eta,\psi(f))}_D\right).
\end{align*}
\end{proof}

\section{Simple records for multivariate observations}\label{sec:simple_records}

\subsection*{Simple record probability}

In the preceding section, we have investigated the (normalized) probability of a complete record and in particular, its limit, the extremal concurrence probability. Now we will repeat this procedure, this time for the simple record probability. Unlike in the previous section, where we were actually dealing with the probability of having a champion, normalizing the record probability with the factor $n$ does not yield an interpretation in terms of a probability in the simple record case.

The following result is the equivalent of Theorem \ref{thm:champions_copula_processes} and Proposition \ref{prop:survival_function_champion} in the context of multivariate simple records. Let $\bfX,\bfX_1,\bfX_2,\dots$ be i.\,i\,d. rv in $\R^d$ with common continuous df $F$. Recall that $\bfX_n$ is a simple record, if
\[
\bfX_n\not\le \max_{1\leq i\le n-1}\bfX_i,
\]
and $\ubar\pi_n(\bm X)$ denotes the probability of $\bm X_n$ being a simple record within the iid sequence $\bm X_1,\bm X_2,\dotsc$

\begin{theorem}\label{theo:asymptotic_distribution_of_records}
Let $\bfU_1,\bfU_2,\dots$ be independent copies of a rv $\bfU\in\R^d$ following a copula $C$. Suppose that $C\in\mathcal D(G)$ with $G(\bfx)=\exp(-\norm{\bfx}_D)$, $\bfx\le\bfzero\in\R^d$. Let $\bm\eta$ be a rv with df $G$. Then
\[
n\ubar\pi_n(\bm U)\to_{n\to\infty}E\left(\norm{\bm\eta}_D\right),
\]
and
\begin{align*}
&P(n(\bfU_n-\bfone)\le\bfx\mid \bfU_n\textnormal{ is a simple record})\\
&\to_{n\to\infty} H_D(\bfx):=\frac{E(\norm{\min(\bfx,\bfeta)}_D)-\norm{\bfx}_D}{E(\norm{\bfeta}_D)},\qquad \bfx\le\bfzero\in\R^d.
\end{align*}
\end{theorem}

In the one dimensional case $d=1$ we obtain $H_D(x)=\exp(x)$, $x\le 0$. Note, however, that $H_D$ is not a probability df in general. Take, for instance, $\norm\cdot_D=\norm\cdot_1$, which is the largest $D$-norm. In this case the components $\eta_1,\dots,\eta_d$ of $\bfeta$ are independent and we obtain for $\bfx=(x_1,\dots,x_d)\le \bfzero\in\R^d$
\begin{equation*}
H_1(\bfx)= \frac{\sum_{i=1}^d \Big(E(\abs{\min(x_i,\eta_i)})-\abs{x_i}\Big)}{\sum_{i=1}^dE(\abs{\eta_i})}
= \frac{\sum_{i=1}^d\exp(x_i)} d.
\end{equation*}
This is not a probability df on $(-\infty,0]^d$ as, for example, $H_1(\bfx)$ does not converge to zero if only one component $x_i$ converges to $-\infty$. Even more, choose $\bfa\le \bfb\le\bfzero\in\R^d$. If $H_1$ would define a probability measure $Q$ on $(-\infty,0]^d$, then the probability $Q([\bfa,\bfb])$
 were given by
\[
\Delta_{\bfa}^{\bfb}H_1=\sum_{\bfm\in\set{0,1}^m}(-1)^{d-\sum_{1\le j\le d} m_j} H_1\left(b_1^{m_1}a_1^{1-m_1},\dots,b_d^{m_d}a_d^{1-m_d}\right).
\]
But elementary computations show that $\Delta_{\bfa}^{\bfb}H_1=0$, i.e., $Q$ is the null measure on $(-\bfinfty,\bfzero]$.

Instead one can define $Q$ on $[-\infty,0]^d\backslash \set{-\bfinfty}$ by putting for $x_i\le 0$ and $i=1,\dots,d$
\[
Q\big(\set{-\infty}\times\dots\times \set{-\infty}\times (-\infty,x_i]\times \set{-\infty}\times \dots \times\set{-\infty}\big):= \frac {\exp(x_i)}d.
\]
 Then $Q$ has its complete mass on the set $\Big\{\bigcup_{i=1}^d\Big(\set{-\infty}^{i-1}\times(-\infty,0]\times \set{-\infty}^{d-i}\Big)\Big\}$ and
 \begin{align*}
 &Q\left(\times_{i=1}^d[-\infty,x_i]\backslash\set{-\bfinfty}\right)\\
 &= Q\left(\bigcup_{i=1}^d  \left(\set{-\infty}\times\dots\times \set{-\infty}\times (-\infty,x_i]\times \set{-\infty}\times \dots \times\set{-\infty}\right)    \right)\\
 &= \sum_{i=1}^d Q\left(\set{-\infty}\times\dots\times \set{-\infty}\times (-\infty,x_i]\times \set{-\infty}\times \dots \times\set{-\infty}\right)\\
 &= \frac 1d \sum_{i=1}^d \exp(x_i).
 \end{align*}
 This approach is closely related to the formulation of the exponent measure theorem as in \citet{balr77} and \citet{vatan85}.

 Take, on the other hand, $\norm\cdot_D=\norm\cdot_\infty$, which is the least $D$-norm. In this case, the components $\eta_1,\dots,\eta_d$ of $\bfeta$ are completely dependent, i.e., $\eta_1=\eta_2=\dots=\eta_d$ a.s. and, thus,
\begin{align*}
 H_\infty(\bfx)&= E\left(\norm{\left(\min(x_i,\eta_1)\right)_{i=1}^d}_\infty\right)- \norm{\bfx}_\infty\\
 &= E\left(\max(\norm{\bfx}_\infty,\eta_1)\right) - \norm{\bfx}_\infty\\
 &= \exp(-\norm{\bfx}_\infty),\qquad \bfx=(x_1,\dots,x_d)\le \bfzero\in\R^d,
\end{align*}
which is a max-stable distribution (MSD).

\begin{proof}[Proof of Theorem \ref{theo:asymptotic_distribution_of_records}]
Let $\bm Z$ be a generator of $\norm\cdot_D$, independent of $\bm\eta$. Theorem \ref{thm:champions_copula_processes}, the inclusion-exclusion principle and \eqref{eq:max=min} yield
\begin{align*}
n\ubar\pi_n(\bm U)&=nP\left(\bm U\not\leq\max_{i=1,\dotsc,n-1}\bm U_i\right)\\
&=nP\left(\bigcup_{j=1}^d\left\{U_{j}>\max_{i=1,\dotsc,n-1}U_{i,j}\right\}\right)\\
&=\sum_{\emptyset\neq T\subset\{1,\dotsc,d\}}(-1)^{\abs{T}-1}nP\left(U_{j}>\max_{i=1,\dotsc,n-1}U_{i,j},~j\in T\right)\\
&\to_{n\to\infty}\sum_{\emptyset\neq T\subset\{1,\dotsc,d\}}(-1)^{\abs{T}-1}E\left(\min_{j\in T}\abs{\eta_j}Z_j\right)\\
&=E\left(\max_{j=1,\dotsc,d}\abs{\eta_j}Z_j\right)\\
&=E\left(\norm{\bm\eta}_D\right).
\end{align*}
Similarily, one can use Proposition \ref{prop:survival_function_champion} in order to show for $\bm x\leq\bm0\in\R^d$
\[
nP\left(n(\bm U-1)\not\leq \min(\bm x,\bm M_n)\right)\to_{n\to\infty}E\left(\norm{\min(\bm x,\bm\eta)}_D\right),
\]
where $\bm M_n:=n\max_{i=1,\dotsc,n-1}\left(\bm U_n-1\right)\to_{\mathcal D}\bm\eta$. In summary, taking into account \eqref{eq:fdoa_copula_process_expansion}, we obtain
\begin{align*}
&nP\left(\bm U\leq1+\frac{\bm x}{n},\bm U\not\leq\max_{i=1,\dotsc,n-1}\bm U_i\right)\\
&\qquad=nP\left(n(\bm U-1)\not\leq\min\left(\bm x,\bm M_n\right)\right)-nP\left(\bm U\not\leq 1+\frac{\bm x}{n}\right)\\
&\qquad\to_{n\to\infty} E\left(\norm{\min(\bm x,\bm\eta)}_D\right)-\norm{\bm x}_D.
\end{align*}
\end{proof}

The arguments in the preceding proof can easily be repeated to extend Theorem \ref{theo:asymptotic_distribution_of_records} to the case of a general rv $\bfX\in\R^d$, whose df is in the domain of attraction of an MSD. Denote by
\[
C_F(\bm u):=F\left(F_1^{-1}(u_1),\dotsc,F_d^{-1}(u_d)\right),\qquad\bm u=(u_1,\dotsc,u_d)\in[0,1]^d,
\]
the copula of a continuous df $F$ on $\R^d$, where $F_i$ is the $i$-th univariate marginal df and $F_i^{-1}$ its quantile function.

\begin{cor}\label{cor:general_case_simple_record_distribution}
´Let $\bfX_1,\bfX_2,\dots$ be independent copies of a rv $\bfX\in\R^d$, whose df $F$ is continuous and its copula $C_F$ satisfies $C_F\in\mathcal D(G)$, $G(\bfx)=\exp(-\norm{\bfx}_D)$, $\bfx\le\bfzero\in\R^d$. We require in addition that each univariate margin $F_i$ of $F$ is in the domain of attraction of a univariate MSD $G_i$, i.e., there are constants $a_{ni}>0$, $b_{ni}\in\R$, $n\in\N$, such that for $i=1,\dots,d$
\[
n(1-F(a_{ni}x+b_{ni}))\to_{n\to\infty} -\log(G_i(x))=:-\psi_i(x),\qquad x\in\R:\, G_i(x)>0.
\]
Then we obtain with $\bfa_n:=(a_{n1},\dots,a_{nd})$, $\bfb_n:=(b_{n1},\dots,b_{nd})$ and $\bm\psi(\bfx):=(\psi_1(x_1),\dots,\psi_d(x_d))$, $\bfx=(x_1,\dots,x_d)$, $G_i(x_i)>0$, $i=1,\dots,d$:
\begin{align*}
P\left(\frac{\bfX_n-\bfb_n}{\bfa_n}\le\bfx\mid \bfX_n\textnormal{ is a simple  record} \right)\to_{n\to\infty} H_D(\bm\psi(\bfx)).
\end{align*}
\end{cor}

Note that in the case $d=1$
\[
H_D(\psi(x))=\exp(\psi(x))=G(x),\qquad G(x)>0.
\]
Note, moreover, that the assumptions on the df $F$ in the preceding theorem are equivalent with the condition $F\in\mathcal D(G)$, where $G$ is a $d$-dimensional MSD, together with the condition that $F$ is continuous.

\begin{proof}[Proof of Corollary \ref{cor:general_case_simple_record_distribution}]
Assume the representation
\[
\bfX=\left(F_1^{-1}(U_1),\dots,F_d^{-1}(U_d)\right)=:\bm F^{-1}(\bfU),
\]
where $\bfU=(U_1,\dots,U_d)$ follows the copula $C$ of $\bfX$. Repeating the arguments in the proof of Theorem \ref{theo:asymptotic_distribution_of_records} now implies the assertion.
\end{proof}

In Corollary \ref{cor:number_of_complete_records}, we have investigated the expected number of complete records as the sample size goes to infinity, which can be done for simple records analogously.

\begin{cor}\label{cor:number_of_simple_records}
Denote by $m(n):=\sum_{i=1}^n 1_{\left\{\bfX^{(i)}\not\leq\max_{1\le j<i}\bfX^{(j)}\right\}}$ the number of simple records among $\bfX^{(1)},\ldots,\bfX^{(n)}$. Then
\[
\frac{E(m(n))}{\log(n)}\to_{n\to\infty} E\left(\norm{\bm\eta}_D\right).
\]
\end{cor}

The next example shows some connection between the Fr\' echet and the Weibull model, and provides in particular closed formulas for $E\left(\dnormf{\bm\eta}_\lambda\right)$ and $E\left(\norm{\bm\eta}_\lambda\right)$.

\begin{exam}[Fr\' echet model]\label{exam:frechet_records}\upshape
Choose  $\lambda>1$. Let $\bm\eta$ be a max-stable rv in $\R^d$ with df $P(\bm\eta\leq\bm x)=\exp\left(-\norm{\bm x}_\lambda\right)$, $\bm x\leq 0$. Let $\bm Z_F$ a Fr\' echet-based generator of $\norm\cdot_\lambda$ (see Example \ref{exam:frechet}), and $\bm Z_W$ a Weibull-based generator of $\norm\cdot_{W_\lambda}$ (see Example \ref{exam:weibull}).  We know from \citet[Example 1]{domribsto15} that
\[
E\left(\dnormf{\bm\eta}_\lambda\right)=\frac{\Gamma(d-1/\lambda)}{(d-1)!\Gamma(1-1/\lambda)}=\prod_{i=1}^{d-1}\left(1-\frac{1}{\lambda i}\right).
\]
Analogously, one can show
\[
E\left(\norm{\bm Z_W}_\lambda\right)=\frac{\Gamma(d+1/\lambda)}{(d-1)!\Gamma(1+1/\lambda)}=\prod_{i=1}^{d-1}\left(1+\frac{1}{\lambda i}\right).
\]
Futhermore, \eqref{eq:dualdnorm_weibull} together with Lemma \ref{lemma:other_formula_for_champion_probability} yields
\[
E\left(\dnormf{\bm Z_F}_{W_\lambda}\right)=E\left(\norm{1/\bm Z_F}_\lambda^{-1}\right)=E\left(\dnormf{\bm\eta}_\lambda\right).
\]
On the other hand, it is easy to see that $\left(\Gamma(1-1/\lambda)\Gamma(1+1/\lambda)\bm Z_W\right)^{-1}$ is also a generator of $\norm\cdot_\lambda$, which yields in turn
\[
E\left(\dnormf{\bm Z_F}_{W_\lambda}\right)=E\left(\dnormf{\bm Z_W}_\lambda\right).
\]
Altogether, we obtain
\[
E\left(\dnormf{\bm\eta}_\lambda\right)=E\left(\dnormf{\bm Z_W}_\lambda\right),
\]
and hence by \eqref{eq:max=min}
\[
E\left(\norm{\bm\eta}_\lambda\right)=E\left(\norm{\bm Z_W}_\lambda\right)=\prod_{i=1}^{d-1}\left(1+\frac{1}{\lambda i}\right).
\]
\end{exam}

\subsection*{Simple record times}

It is well known that the record times have infinite expectation for a sequence of univariate iid rv with a common continuous df. This is no longer true in the multivariate case. In this section we give a precise characterization.

Let $\bfX_1,\bfX_2,\dots$ be i.\,i\,d. rv in $\R^d$ with common continuous df $F$. We denote by $N(n)$, $n\ge 1$, the \emph{simple record times}, i.e., those subsequent random indices at which a simple record occurs. Precisely, $N(1)=1$, as $\bfX_1$ is, clearly, a simple record, and, for $n\ge 2$,
\[
N(n):=\min\set{j:\;j>N(n-1),\,\bfX_j\not\le\max_{1\le i\le N(n-1)}\bfX_i}.
\]

As the df $F$ is continuous, the distribution of $N(n)$ does not depend on $F$ and, therefore, we assume in what follows without loss of generality that $F$ is a \emph{copula} $C$ on $\R^d$, i.e., each component of $\bfX_i$ is on $(0,1)$ uniformly distributed.

Conditioning on $\bm X_1=\bm u$ yields for $j\ge 2$
\begin{align*}
P(N(2)=j) &= P(\bfX_2\le \bfX_1,\dots,\bfX_{j-1}\le \bfX_1,\bfX_j\not\le \bfX_1)\\
&= \int_{[0,1]^d} C(\bfu)^{j-2}(1-C(\bfu))\,C(\d\bfu).
\end{align*}
Solving the geometric series, we get
\begin{equation}\label{eqn:expectation_of_N(2)}
E(N(2))= \sum_{j=2}^\infty j P(N(2)=j)=\int_{[0,1]^d} \frac{1}{1-C(\bfu)}\,C(\d\bfu) + 1.
\end{equation}
Now we generalize this formula. Choose $n\in\N$. Partitioning the sample space in disjoint events, we obtain for $k_n\geq 1$
\begin{align*}
&P\left(N(n+1)-N(n)=k_n\right)=\\
&=\sum_{k_1=2}^{\infty}\sum_{k_2=1}^{\infty}\cdots\sum_{k_{n-1}=1}^\infty P\left(N(n+1)=\sum_{j=1}^{n}k_j,N(n)=\sum_{j=1}^{n-1}k_j,\dotsc,N(2)=k_1\right),
\end{align*}
and further, similar to the calculation above,
\begin{align*}
&P\left(N(n+1)=\sum_{j=1}^{n}k_j,N(n)=\sum_{j=1}^{n-1}k_j,\dotsc,N(2)=k_1\right)=\\
&\int_{\{\bm u_n\not\leq\cdots\not\leq\bm u_1\}}C(\bm u_1)^{k_1-2}C(\bm u_2)^{k_2-1}\cdots C(\bm u_n)^{k_n-1}(1-C(\bm u_n))~ C(\d\bm u_1)\cdots C(\d\bm u_n).
\end{align*}
Hence, solving all the occuring geometric series yields
\begin{align*}
&E\left(N(n+1)-N(n)\right)=\int_{\{\bm u_n\not\leq\cdots\not\leq\bm u_1\}}\prod_{i=1}^n\frac1{1-C(\bm u_i)}~ C(\d\bm u_1)\cdots C(\d\bm u_n).
\end{align*}
Furthermore, it is easy to see that for all $k\geq 1$
\begin{equation}\label{eq:geometric_distribution}
P\left(N(n+1)-N(n)=k|\bm X_{N(n)}=\bm u\right)=C(\bm u)^{k-1}(1-C(\bm u)),
\end{equation}
which means that $N(n+1)-N(n)|\bm X_{N(n)}=\bm u$ is geometrically distributed with parameter $1-C(\bm u)$. We summarize these results.

\begin{lemma}\label{lem:distribution_of_N(n)}
Let $\bm X_1,\bm X_2,\dotsc$ be iid rv following a copula $C$ on $[0,1]^d$, and denote by $N(n)$ the $n$-th simple record time, $n\in\N$.
\begin{enumerate}[(i)]
\item
For every $n\in\N$
\[
N(n+1)-N(n)|\bm X_{N(n)}=\bm u\sim\Geom(1-C(\bm u)),
\]
where $\Geom(p)$ denotes the geometric distribution with support $\{1,2,\dotsc\}$ and parameter $p\in(0,1]$.
\item For every $n\in\N$,
\begin{equation}\label{eq:expectation_of_N(n)}
E\left(N(n+1)-N(n)\right)=\int_{\{\bm u_n\not\leq\cdots\not\leq\bm u_1\}}\prod_{i=1}^n\frac1{1-C(\bm u_i)}~ C(\d\bm u_1)\cdots C(\d\bm u_n).
\end{equation}
\end{enumerate}
\end{lemma}

Suppose now that $d=1$. Then we have $\bfu=u\in[0,1]$, $C(u)=u$ and
\[
E(N(2)) = \int_0^1 \frac 1{1-u}\,du+1=\infty,
\]
which is well-known (\citet[Theorem 6.2.1]{gal87}).

Suppose next that $d\ge 2$ and that the margins of $C$ are independent, i.e.,
\[
C(\bfu) = \prod_{i=1}^d u_i,\qquad \bfu=(u_1,\dots,u_d)\in[0,1]^d.
\]
Then we obtain
\begin{equation*}
\int_{[0,1]^d} \frac{1}{1-C(\bfu)}\,C(\d\bfu) = \int_0^1\dots\int_0^1 \frac{1}{1-\prod_{i=1}^d u_i}\, \d u_1\dots \d u_d <\infty
\end{equation*}
by elementary arguments and, thus, $E(N(2))<\infty$. This observation gives rise to the problem of characterizing those copulas $C$ on $[0,1]^d$ with $d\ge 2$, such that $E(N(2))$ is finite. Note that $E(N(2))=\infty$ if the components of $C$ are completely dependent.

The next lemma characterizes finiteness of $E(N(2))$.

\begin{lemma}\label{lem:characterization_of_E(N(2))}
Let $\bfX=(X_1,\dots,X_d)$ follow a copula $C$ on $\R^d$. Then $E(N(2))<\infty$ iff
\begin{equation}\label{eqn:characterization_of_E(N(2))}
\int_1^\infty P\left(X_i> 1-\frac 1 t,\,1\le i\le d\right)\,\d t < \infty.
\end{equation}
\end{lemma}

Condition \eqref{eqn:characterization_of_E(N(2))} is trivially satisfied in case of independent components $X_1,\dots,X_d$ and $d\ge 2$. Below we will see that it is, roughly, in general satisfied, if there are at least two components that are asymptotically independent.

\begin{proof}[Proof of Lemma \ref{lem:characterization_of_E(N(2))}]
Any copula $C$ satisfies the Fr\' echet-Hoeffding bounds, that is, for $\bm u=(u_1,\dotsc,u_d)\in[0,1]^d$,
\begin{equation}\label{eq:frechet_hoeffding}
\max\left(1-d+\sum_{i=1}^du_i,0\right)\leq C(\bm u)\leq\min\left(u_1,\dotsc,u_d\right).
\end{equation}
Therefore, we obtain due to the upper bound in \eqref{eq:frechet_hoeffding}
\begin{align*}
E(N(2))-1&=\int_{[0,1]^d}\frac{1}{1-C(\bfu)}\, C(\d\bfu)\\
&=E\left(\frac1{1-C(\bm X)}\right)\\
&=\int_1^\infty P\left(C(\bm X)>1-\frac1t\right)~\d t\\
&\leq\int_1^\infty P\left(X_i>1-\frac1t,1\leq i\leq d\right)~\d t.
\end{align*}
On the other hand, the lower bound in \eqref{eq:frechet_hoeffding} yields
\begin{align*}
E(N(2))-1&\geq \int_1^\infty P\left(\sum_{i=1}^d\left(1-X_i\right)<\frac1t \right)~\d t\\
&\geq \int_1^\infty P\left(1-X_i<\frac1{dt},1\leq i\leq d \right)~\d t\\
&=d\int_{1/d}^\infty P\left(1-X_i<\frac1{t},1\leq i\leq d \right)~\d t.
\end{align*}

\end{proof}

Let $C$ be a copula that is in the domain of attraction of a (standard) max-stable df $G$ on $\R^d$, i.\,e.
\[
C^n\left(1+\frac{\bm x}{n}\right)\to_{n\to\infty} G(\bm x),\qquad \bm x\leq0.
\]
which we abbreviate by $C\in\mathcal D(G)$. Recall that Proposition \ref{prop:doa_copula_process} also applies to the multivariate case, see also Remark \ref{rem:multivariate_doa}.

\begin{prop}\label{prop:expectation_of_N(2)_is_infinite}
Suppose that $C\in\mathcal D(G)$, where the $D$-norm corresponding to $G$ satisfies $\dnormf{\bfone}_D>0$. Then $E(N(2))=\infty$.
\end{prop}

\begin{proof}
Let $\bm X$ be a rv with df $C$. It is well known from real analyis that
\[
\int_1^\infty P\left(\bm X> 1-\frac 1 t\right)\,\d t < \infty\iff \sum_{n=1}^\infty P\left(\bm X>1-\frac1n\right)<\infty.
\]
From \eqref{eq:survival_function_expansion}, we know $nP\left(\bm X>1-\frac1n\right)\to_{n\to\infty}\dnormf{\bm 1}_D>0$. Now applying the limit comparison test for an infinite series, we deduce that $\sum_{n=1}^\infty P\left(\bm X>1-\frac1n\right)$ has the same limit behaviour as the harmonic series $\sum_{i=1}^\infty\frac1n$, and hence, $E(N(2))=\infty$ by Lemma \ref{lem:characterization_of_E(N(2))}.

\end{proof}

Suppose that $C\in\mathcal D(G)$. A finite expectation $E(N(2))<\infty$ can, therefore, only occur if $\dnormf{\bfone}_D=0$, which is true, for instance, if $G$ has at least two independent margins.

Let $\bm X$ follow the df $C$. Next we show that $E(N(2))$ is typically finite if $\bm X$ has at least two components $X_j,X_k$ which are tail independent, i.\,e.
\[
\lim_{u\uparrow 1}P(X_k>u|X_j>u)=0.
\]
In case the limit exists, define the dependence measure
\[
\bar\chi:= \lim_{u\uparrow 1}\frac{ 2\log(1-u)}{\log(P(X_1>u,X_2>u))} - 1\in[-1,1],
\]
where $(X_1,X_2)$ follows some bivariate copula, cf. \citet{colht99, hef00}. Note that we have $\bar\chi=1$ if $X_1,X_2$ are tail dependent. In the class of (bivariate) copulas that are tail independent, however, $\bar\chi$ is a popular measure of tail comparison. For a bivariate normal copula with coefficient of correlation $\rho\in(-1,1)$ it is, for instance, well known that $\bar\chi=\rho$.

\begin{prop}\label{prop:finiteness_of_E(N(2))}
Let $\bfX=(X_1,\dots,X_d)$ follow a copula $C$ in $\R^d$. Suppose that there exist indices $k\not=j$ such that
\[
\bar \chi_{k,j}= \lim_{u\uparrow 1}\frac{ 2\log(1-u)}{\log(P(X_k>u,X_j>u))} - 1\in(-1,1).
\]
Then we have $E(N(2))<\infty$.
\end{prop}

\begin{cor}
We have $E(N(2))<\infty$ for multivariate normal rv unless all components are completely dependent.
\end{cor}

\begin{proof}[Proof of Proposition \ref{prop:finiteness_of_E(N(2))}] We have
\begin{align*}
&\int_{1}^\infty P\left(X_i\ge 1-\frac 1 t,\,1\le i\le d\right)\,\d t\\
&\le \int_{1}^\infty P\left(X_k\ge 1-\frac 1 t,\,X_j \ge 1-\frac 1 t\right)\,\d t\\
&= \int_{1}^\infty \exp\left(\frac{\log\left(P\left(X_k\ge 1-\frac 1 t,\,X_j \ge 1-\frac 1 t\right) \right)}{\log\left(\frac 1 {t^2}\right)}\log\left(\frac 1 {t^2}\right)\right)\,\d t,
\end{align*}
where
\[
\frac{\log\left(P\left(X_k\ge 1-\frac 1 t,\,X_j \ge 1 -\frac 1 t\right) \right)}{\log\left(\frac 1 {t^2}\right)} \to_{t\to\infty} \frac 1{1+\bar\chi} > \frac 1 2.
\]
But this implies that the above integral is finite and, thus, the assertion is a consequence of Lemma \ref{lem:characterization_of_E(N(2))}.
\end{proof}

Next we investigate $E(N(n))$ for $n\geq 2$. As before, let $\bm X_1,\bm X_2,\cdots$ be an iid sequence of rv on $\R^d$ following a copula $C$.  Obviously, $E(N(2))=\infty$ implies $E(N(n))=\infty$ for $n\ge 2$, since $N(n)\ge N(2)$, $n\ge 2$. On the other hand, if $E(N(2))<\infty$, we obtain due to \eqref{eq:expectation_of_N(n)} for all $n\geq 2$
\begin{align*}
E\left(N(n+1)-N(n)\right)&=E\left(\frac{1}{1-C(\bm X_1)}\cdots\frac1{1-C(\bm X_n)}1_{\{\bm X_n\not\leq\cdots\not\leq\bm X_1\}}\right)\\
&\leq E\left(\frac{1}{1-C(\bm X_1)}\cdots\frac1{1-C(\bm X_n)}\right)\\
&\leq\left[E\left(\frac{1}{1-C(\bm X_1)}\right)\right]^n\\
&=\left[E\left(N(2)\right)-1\right]^n,
\end{align*}
which means that $E(N(n+1)-N(n))<\infty$ as well in that case. Furthermore, we will show below that $E(N(n+1)-N(n))=\infty$ for all $n\in\N$ if $E(N(2))=\infty$. In conclusion, it is sufficient to decide whether $E(N(2))$ is finite or not if expecations of arbitrary simple record times are investigated. We summarize this discussion.

\begin{prop}
Let $\bm X_1,\bm X_2,\dotsc$ be iid rv following a copula $C$ on $[0,1]^d$. Then the following implications hold:
\begin{enumerate}[(i)]
\item If $E(N(2))=\infty$, then $E(N(n+1)-N(n))=\infty$ for all $n\in\N$.
\item If $E(N(2))<\infty$, then $E(N(n+1)-N(n))<\infty$ for all $n\in\N$.
\end{enumerate}
\end{prop}

\begin{proof}
It remains to proof (ii). We show that $N(n+1)-N(n)$, $n\in\N$, is \emph{stochastically increasing}, i.\,e. for every $t\in\R$ and every $n\in\N$
\begin{equation}\label{eq:stochastically_increasing}
P(N(n+1)-N(n)\leq t)\geq P(N(n+2)-N(n+1)\leq t),\qquad t\in\R,~n\in\N.
\end{equation}
Recall that the df of a rv $X\sim\Geom(p)$ is given by $P(X\leq t)=1-(1-p)^{\lfloor t\rfloor}$, where $\lfloor t\rfloor=\max\{m\in \Z:~m\leq x\}$.  Conditioning on $\bm X_{N(n)}=\bm x$, we obtain by Lemma \ref{lem:distribution_of_N(n)} (i) for each $t\in\R$ and $n\in\N$
\begin{align*}
&P(N(n+1)-N(n)\leq t)\\
&=\int_{[0,1]^d}P\left(N(n+1)-N(n)\leq t|\bm X_{N(n)}=\bm x\right)~(P*\bm X_{N(n)})(\d\bm x)\\
&=\int_{[0,1]^d}1-C(\bm x)^{\lfloor t\rfloor}~(P*\bm X_{N(n)})(\d\bm x)\\
&=1-E\left(C\left(\bm X_{N(n)}\right)^{\lfloor t\rfloor}\right),
\end{align*}
which shows \eqref{eq:stochastically_increasing} since $\bm X_{N(n)}\leq\bm X_{N(n+1)}$. Hence,
\begin{align*}
E(N(n+1)-N(n))&=\int_0^\infty P(N(n+1)-N(n)>t~\d t)\\
&\leq \int_0^\infty P(N(n+2)-N(n+1)>t~\d t)\\
&=E(N(n+2)-N(n+1)).
\end{align*}
\end{proof}

Computing the distribution of the second record $\bfX_{N(2)}$ is an easy task: Let $\bfX_1,\bfX_2\dots$ be independent copies of the rv $\bfX$ in $\R^d$ with df $F$. We make no further assumption on $F$. Conditioning on $\bfX_1=\bfy$ we obtain
\begin{align*}
P\left(\bfX_{N(2)}\le \bfx\right) &= \sum_{j=2}^\infty P(\bfX_j\le \bfx, N(2)=j)\\
&= \sum_{j=2}^\infty \int_{\R^d}P(\bfX_j\le \bfx, \bfX_2\le \bfy,\dots,\bfX_{j-1}\le\bfy,\bfX_j\not\le \bfy)\,F(\d\bfy)\\
&= \sum_{j=2}^\infty \int_{\R^d} F(\bfy)^{j-2}P(\bfX_j\le \bfx,\bfX_j\not\le \bfy)\,F(\d\bfy)\\
&= \int_{\R^d} P(\bfX\le \bfx,\bfX\not\le \bfy) \sum_{j=2}^\infty F(\bfy)^{j-2}\, F(\d\bfy)\\
&= \int_{\R^d} \frac{P(\bfX\le \bfx,\bfX\not\le \bfy)}{1- F(\bfy)}\, F(\d\bfy)\\
&= \int_{\R^d}P(\bfX\le \bfx\mid\bfX\not\le \bfy)\, F(\d\bfy).
\end{align*}

Computation of the distribution of $\bfX_{N(k)}$ for an arbitrary $k\ge 2$ is much more complex, but manageable. Computation of the limit distribution of $\bm X_{N(k)}$, properly linearly standardized, as $k$ tends to infinity, is an open problem. For the univariate case we refer to \citet[Section 6.4]{gal87}.

\end{document}